\documentclass{amsart}
\usepackage{amscd,graphicx,epsfig}
\usepackage{amssymb}

\title[Lie subalgebras of vector fields and the Jacobian Conjecture]{{\normalfont\rm Lie subalgebras of vector fields and the Jacobian Conjecture}}
\author{Andriy Regeta}
\date{September, 2013}
\thanks{The author is supported by a grant from the SNF (Schweizerischer Nationalfonds)}

\address{Mathematisches Institut,
Universit\"at Basel, Rheinsprung 21, CH-4051 Basel}
\email{Andriy.Regeta@unibas.ch}

\newtheorem{thm}{Theorem}[section]
\newtheorem*{thm*}{Theorem}
\newtheorem*{conj*}{Conjecture}

\newtheorem*{prob*}{Problem}
\newtheorem*{satz*}{Satz}

\newtheorem{prop}{Proposition}[section]
\newtheorem*{prop*}{Proposition}
\newtheorem{lem}[prop]{Lemma}
\newtheorem*{lem*}{Lemma}
\newtheorem{cor}[prop]{Corollary}
\newtheorem*{cor*}{Corollary}

\theoremstyle{definition}
\newtheorem{defn}[prop]{Definition}

\theoremstyle{remark}
\newtheorem*{rem*}{Remark}
\newtheorem{rem}[prop]{Remark}

\newcommand{\ddx}{\frac{\partial}{\partial x}}
\newcommand{\ddy}{\frac{\partial}{\partial y}}

\newcommand{\dx}{\partial_{x}}
\newcommand{\dy}{\partial_{y}}

\newcommand{\CC}{{\mathbb C}}

\newcommand{\Atwo}{{\mathbb A}^{2}}
\newcommand{\Aone}{{\mathbb A}^{1}}

\newcommand{\simto}{\xrightarrow{\sim}}
\newcommand{\be}{\begin{enumerate}}
\newcommand{\ee}{\end{enumerate}}

\DeclareMathOperator{\End}{End}
\DeclareMathOperator{\id}{id}

\DeclareMathOperator{\SLtwo}{SL_{2}}

\DeclareMathOperator{\Aut}{Aut}
\DeclareMathOperator{\SAut}{SAut}

\DeclareMathOperator{\VF}{Vec}
\DeclareMathOperator{\GL}{GL}
\DeclareMathOperator{\Lie}{Lie}
\newcommand{\VFd}{\VF^{c}}
\newcommand{\VFzero}{\VF^{0}}
\DeclareMathOperator{\Der}{Der}
\DeclareMathOperator{\Jac}{Jac}
\DeclareMathOperator{\Div}{div}
\DeclareMathOperator{\cent}{\mathfrak{cent}}
\DeclareMathOperator{\rad}{\mathfrak{rad}}

\newcommand{\px}{\frac{\partial }{\partial x}}
\newcommand{\py}{\frac{\partial }{\partial y}}

\newcommand{\AutK}{\Aut(K[x,y])}
\newcommand{\EndK}{\End(K[x,y])}
\newcommand{\SAutK}{\SAut(K[x,y])}
\newcommand{\AutLA}{\Aut_{\text{\it LA}}}
\newcommand{\SAutLA}{\SAut_{\text{\it LA}}}

\newcommand{\bbmat}{\begin{bmatrix}}
\newcommand{\ebmat}{\end{bmatrix}}
\newcommand{\bsmat}{\begin{smallmatrix}}
\newcommand{\esmat}{\end{smallmatrix}}

\DeclareMathOperator{\aff}{\mathfrak{aff}}
\DeclareMathOperator{\saff}{\mathfrak{saff}}
\DeclareMathOperator{\sltwo}{\mathfrak{sl}_{2}}
\DeclareMathOperator{\Aff}{Aff}
\DeclareMathOperator{\SAff}{SAff}
\DeclareMathOperator{\tdeg}{tdeg}

\newcommand{\s}{\mathfrak{s}}

\newcommand{\Pb}{{\bar P}}
\newcommand{\Ptwo}{{P_{\leq 2}}}
\newcommand{\Pbtwo}{{\bar{P}_{\leq 2}}}

\renewcommand{\phi}{\varphi}
\def \itt #1,#2:{\medskip\item[$\bullet$] %
     page\ \ignorespaces#1, line\ \ignorespaces#2:\ \ignorespaces}

\newcommand{\lab}[1]{\label{#1}}

\begin{document}

\begin{abstract} We study Lie subalgebras $L$ of the vector fields $\VFd(\Atwo)$ of affine 2-space $\Atwo$ of constant divergence, and we classify those $L$ which are isomorphic to the Lie algebra $\aff_{2}$ of the group $\Aff_{2}(K)$ of affine transformations of $\Atwo$. We then show that the following statements are equivalent:
\be
\item The Jacobian Conjecture holds in dimension 2;
\item All Lie subalgebras $L \subset \VFd(\Atwo)$ isomorphic to $\aff_{2}$ are conjugate under $\Aut(\Atwo)$;
\item All Lie subalgebras $L \subset \VFd(\Atwo)$ isomorphic to $\aff_{2}$  are algebraic.
\ee
\end{abstract}

\maketitle
\section{Introduction}

It is a well-known consequence of the amalgamated product structure of $\Aut(\Atwo)$ that every reductive subgroup $G \subset \Aut(\Atwo)$ is conjugate to a subgroup of $\GL_{2}(\CC) \subset \Aut(\Atwo)$, i.e.  there is a $\psi \in \Aut(\Atwo)$ such that $\psi G \psi^{-1} \subset \GL_2(\CC)$ (\cite{Ka1979Automorphism-group}, cf. \cite{Kr1996Challenging-proble}). The ``Linearization Problem'' asks whether the same holds in higher dimension.
It was shown by Schwarz in \cite{Sch89} that this is not the case in dimensions $n \ge 4$ .

In this paper we consider the analogue of the Linearization Problem for Lie algebras. It is known that the Lie algebra $\Lie(\Aut(\Atwo))$ of the ind-group $\Aut(\Atwo)$ is isomorphic to the Lie algebra $\VF^c(\Atwo)$ of vector fields  of constant divergence (\cite{Sh81}, cf. \cite{Kum02}). 
We will see  that  the Lie subalgebra 
$$
K(x^2 \px - 2xy \py) \oplus K (x\px - y \py) \oplus K\px \subset \VF^{c}(\Atwo)
$$
is isomorphic to $\sltwo$, but not conjugate to the standard $\sltwo \subset \VF(\Atwo)$ under 
$\Aut(\Atwo)$ (Remark~\ref{example.rem}). On the other hand, for other subalgebras of $\VF^c(\Atwo)$ the situation is different.
Let  $\Aff_2(K)\subset\Aut(\Atwo)$ be  the group of affine transformations and  $\SAff_2(K)$ the subgroup of affine transformation  with determinant equal to 1, and denote by $\aff_{2}$, resp. $\saff_{2}$ their Lie algebras. The first result we prove is the following (see Proposition~\ref{subVec2.prop}). 
For $f \in K[x,y]$ we set $D_{f}:=-f_{y}\ddx + f_{x}\ddy \in\VF(\Atwo)$.

\begin{prop*}
Let $L\subset \VFd(\Atwo)$ be a Lie subalgebra isomorphic to  $\aff_2$. Then there is an \'etale map $\alpha=(f,g)$ such that $L=\alpha(\aff_2)$. More precisely, if $(D_{f},D_{g})$ is a basis of the radical of $[L,L]$, then 
$$
L = \langle D_{f},D_{g},D_{f^{2}}, D_{g^{2}}, fD_g,gD_f \rangle,
$$ 
and one can take $\alpha=(f,g)$.
\end{prop*}
The analogous statements hold for Lie subalgebras isomorphic to $\saff_2$.
As a consequence of this classification we obtain the next result
(see  Theorem~\ref{JC.thm} and Corollary~\ref{corollary}). Recall that a Lie subalgebra of $\VF(\Atwo)$ is {\it algebraic\/} if it acts locally finitely on $\VF(\Atwo)$.
\begin{prop*} The following statements are equivalent:
\be
\item[(i)] The Jacobian Conjecture holds in dimension 2;
\item[(ii)] All Lie subalgebras $L \subset \VFd(\Atwo)$ isomorphic to $\aff_{2}$ are conjugate under $\Aut(\Atwo)$;
\item[(iii)] All Lie subalgebras $L \subset \VFd(\Atwo)$ isomorphic to $\saff_2$ are conjugate under $\Aut(\Atwo)$;
\item[(iv)] All Lie subalgebras $L \subset \VFd(\Atwo)$ isomorphic to $\aff_2$  are algebraic;
\item[(v)] All Lie subalgebras $L \subset \VFd(\Atwo)$ isomorphic to $\saff_2$  are algebraic;
\ee
\end{prop*}

{\small\noindent
\textbf{Acknowledgement}:
The author would like to thank his thesis advisor Hanspeter Kraft for permanent support and help during writing this paper.}

\medskip
\section{The Poisson algebra}
\subsection*{Definitions}
Let $K$ be an algebraically closed field of characteristic zero and let $P$ be the {\it Poisson algebra}, i.e. the Lie algebra with underlying vector space $K[x,y]$ and with Lie bracket $\{f,g\}:=f_{x}g_{y} - f_{y}g_{x}$ for $f,g\in P$. If $\Jac(f,g)$ denotes the {\it Jacobian matrix\/} and  $j(f,g)$ the {\it Jacobian determinant},
$$
\Jac(f,g): = \begin{bmatrix} f_x & f_y \\ g_x &g_y \end{bmatrix},\quad j(f,g) := \det \Jac(f,g),
$$
then $\{f,g\} = j(f,g)$. Denote by $\VF(\Atwo)$ the polynomial vector fields on affine 2-space $\Atwo = K^{2}$, i.e. the derivations of $K[x,y]$:
$$
\VF(\Atwo) := \{ p\dx + q\dy \mid p,q \in K[x,y]\} = \Der(K[x,y]).
$$
There is a canonical homomorphism of Lie algebras
$$
\mu\colon P \to \VF(\Atwo), \ h\mapsto D_{h}:=h_{x}\dy - h_{y}\dx,
$$
with kernel $\ker\mu = K$.

The next lemma collects some properties of the Lie algebra $P$. These results are essentially known, see e.g. \cite{NoNa1988Rings-of-constants}.
If $L$ is any Lie algebra and $X \subset L$ a subset, we define the {\it centralizer\/} of $X$ by
$$
\cent_{L}(X) := \{z \in L \mid [z,x]=0 \text{ for all } x\in X\},
$$
and we shortly write $\cent(L)=\cent_{L}(L)$ for the {\it center\/} of $L$.
\begin{lem}\lab{lem1}
\be
\item The center of $P$ are the constants $K$.
\item Let  $f,g\in P$ such that $\{f,g\}= 0$. Then $f,g\in K[h]$ for some $h\in K[x,y]$.
\item If  $f,g\in P$ such that $\{f,g\}\neq 0$, then $f,g$ are algebraically independent in $K[x,y]$ and $\cent_{P}(f) \cap \cent_{P}(g) = K$.
\item $P$ is generated, as a Lie algebra, by $\{x, x^{3},y^{2}\}$.
\ee
\end{lem}
\begin{proof} (a) is easy and left to the reader. 

(b) Consider the morphism $\alpha=(f,g)\colon \Atwo \to \Atwo$. Then $C:=\overline{\alpha(\Atwo)} \subset \Atwo$ is an irreducible rational curve, and we have a factorization
$$
\begin{CD}
\alpha\colon \Atwo @>h>> \Aone @>\eta>> C \subset \Atwo
\end{CD}
$$
where $\eta$ is the normalization of $C$. It follows that $f,g \in K[h]$.

(c) It is clear that $f,g$ are algebraically independent, i.e. $\tdeg_{K}K(f,g) = 2$. Equivalently, $K(x,y) / K(f,g)$ is a finite algebraic extension. Now assume that $\{h,f\}=\{h,g\}=0$. Then the derivation $D_{h}$ vanishes on $K[f,g]$, hence on $K[x,y]$. Thus $D_{h}=0$ and so  $h\in K$.

(d) Denote by $P_{d}:=K[x,y]_{d}$ the homogeneous part of degree $d$. Let $L \subset P$ be the Lie subalgebra generated by $\{x, x^{3},y^{2}\}$. We first use the equations
$$
\{x,y\} = 1, \ \{x,y^{2}\}= 2y, \ \{x^{3},y\} = 3x^{2}, \  \{x^{2},y^{2}\}=4xy, \ \{x^{3},y^{2}\}=6x^{2}y
$$
to show that $K\oplus P_{1}\oplus P_{2} \subset L$ and that $x^{2}y\in L$. Now the claim follows by induction from the relations
$$
\{x^{n},x^{2}y\}=nx^{n+1} \text{ \ and \ } \{x^{r}y^{s},y^{2}\}= 2rx^{r-1}y^{s+1}.
$$
\end{proof}

\subsection*{Divergent}
The next lemma should also be known. Recall that the {\it divergence\/}  $\Div D$ of a vector field $D = p \dx + q \dy\in\VF(\Atwo)$ is defined  by $\Div D := p_{x}+q_{y} \in K[x,y]$. 
\begin{lem}\lab{derivation.lem}
Let $D$ be a non-trivial derivation of $K[x,y]$.
\be
\item 
The kernel $K[x,y]^{D}$ is either $K$ or $K[f]$ for some $f\in K[x,y]$.
\item
If $\Div D =0$, then $D=D_{h}$ for some $h\in K[x,y]$.
\ee
Now assume that $D = D_{f}$ for some non-constant $f\in K[x,y]$ and that $D(g)=1$ for some $g\in K[x,y]$.
\be\setcounter{enumi}{2}
\item
Then $K[x,y]^{D}= K[f]$.
\item 
If $D$ is locally nilpotent, then $K[x,y]=K[f,g]$.
\ee
\end{lem}
\begin{proof}
(a) See \cite{NoNa1988Rings-of-constants} Theorem 2.8.

(b) Let $D = f \px + g \py$, then $\Div D = f_x + g_y=0$ implies that there exist $h \in K[x,y]$ such that $f = h_y$, $g = -h_x$.

(c) It is obvious that $\ker(D) \supset K[f]$, hence by (a) one has $\ker(D) = K[h] \supset K[f]$. Thus $f = F(h)$ for some $F \in K[t]$
and then $D_f(g) = D_{F(h)}(g) = F'(h) D_h(g) = 1$ which implies $F$ is linear.

(d) Let $G$ be an affine algebraic group, $X$ be an affine variety and $\phi: X \to G$ be a $G$-equivariant retraction. Then there is the general fact which says that $O(X) = \phi^*(O(G)) \otimes O(X)^G$. In our case $O(\Atwo) =O(G) \otimes O(\Atwo)^G = K[g] \otimes K[f]$.
\end{proof}
It follows from Lemma~\ref{derivation.lem}(b) above that the image of 
$\mu\colon P \to \VF(\Atwo)$, $h\mapsto D_{h}$, is 
$\mu(P) = \VFzero(\Atwo):=\{D\in\VF(\Atwo)\mid \Div D  = 0\}$. We will also discuss the Lie subalgebra $\VFd(\Atwo):=\{D\in\VF(\Atwo)\mid \Div D  \in K\}$.

\subsection*{Automorphisms of the Poisson algebra}
Denote by $\AutLA(P)$ the group of Lie algebra automorphisms of $P$. There is a canonical homomorphism
$$
p\colon \AutLA(P) \to K^{*}, \quad \phi \mapsto \phi(1),
$$
which has a section $s \colon K^{*}\to \AutLA(P)$ given by $s(t)|_{K[x,y]_{n}}:= t^{1-n} \id_{K[x,y]_{n}}$.
Thus $\AutLA(P)$ is a semidirect product $\AutLA(P) = \SAutLA(P) \rtimes K^{*}$ where 
$$
\SAutLA(P) :=\ker p = \{\phi\mid \phi(1)=1\}.
$$
\begin{lem}\lab{lem2}
Every automorphism $\phi\in\AutLA(P)$ is determined by $\phi(1)$, $\phi(x)$  and $\phi(y)$, and $K[x,y] = K[\phi(x),\phi(y)]$.
\end{lem}
\begin{proof} 
Replacing $\phi$ by the composition $\phi\circ s(\phi(1)^{-1})$ we can assume that $\phi(1)=1$. 

We will show that $\phi(x^{n}) = \phi(x)^{n}$ and $\phi(y^{n}) = \phi(y)^{n}$ for all $n\geq 0$. Then the first claim follows from Lemma~\ref{lem1}(d).

By induction, we can assume that $\phi(x^{j}) = \phi(x)^{j}$ for $j <n$. We have $\{x^{n},y\}=nx^{n-1}$ and so $\{\phi(x^{n}),\phi(y)\}=n\phi(x^{n-1}) = n\phi(x)^{n-1}$. On the other hand, we get $\{\phi(x)^{n},\phi(y)\} = n\phi(x)^{n-1}\{\phi(x),\phi(y)\} = n\phi(x)^{n-1}$, hence the difference $h:=\phi(x^{n})-\phi(x)^{n}$ belongs to the kernel of the derivation $D_{\phi(y)}\colon f \mapsto \{f,\phi(y)\}$. Since $D_{\phi(y)}$ is locally nilpotent, we get from Lemma~\ref{derivation.lem}(c)--(d) that $\ker D_{\phi(y)} = K[\phi(y)]$ and that $K[\phi(x),\phi(y)] = K[x,y]$. This already proves the second claim and shows that $h$ is a polynomial in $\phi(y)$. 

Since $\{\phi(x^{n}),\phi(x)\} = \phi(\{x^{n},x\})=0$ and $\{\phi(x)^{n},\phi(x)\}=n\phi(x)^{n-1}\{\phi(x),\phi(x)\}$ we get $\{h,\phi(x)\}=0$ which implies that $h\in K$.

In the same way, using $\{x,xy\} = x$ and $\{y,xy\}=-y$, we find $\phi(xy)-\phi(x)\phi(y) \in K$. Hence
$$
n\phi(x^{n}) = \{\phi(x^{n}),\phi(xy)\}=\{\phi(x)^{n},\phi(x)\phi(y)\}=n\phi(x)^{n},
$$
and  so $\phi(x^{n}) = \phi(x)^{n}$. By symmetry, we also get $\phi(y^{n}) = \phi(y)^{n}$. 
\end{proof}

\subsection*{Automorphisms of affine 2-space}
Denote by $\AutK$ the group of $K$-algebra automorphisms of $K[x,y]$. For $\alpha\in\AutK$ we will use the notation $\alpha=(f,g)$ in case $\alpha(x)=f$ and $\alpha(y)=g$, which implies that  $K[x,y]=K[f,g]$. There is a homomorphism
$$
j\colon \AutK \to K^{*}, \quad \alpha\mapsto j(\alpha):= j(\alpha(x),\alpha(y)) = \det \Jac(\alpha(x),\alpha(y))
$$
which has a section $\sigma\colon t\mapsto (tx,ty)$. Hence, $\AutK$ is a semidirect product
$\AutK = \SAutK \rtimes K^{*}$ where 
$$
\SAutK :=\ker j = \{\alpha=(f,g)\mid j(f,g)=1\}.
$$
We can consider $\AutK$  and $\AutLA(P)$ as subgroups of the $K$-linear automorphisms $\GL(K[x,y])$.
\begin{lem}\lab{aut.lem}
As subgroups of $\GL(K[x,y])$ we have $\SAutLA(P) = \SAutK$.
\end{lem}
\begin{proof}
(a) Let $\alpha$ be an endomorphism of $K[x,y]$ and put $\Jac(\alpha):=\Jac(\alpha(x),\alpha(y))$. For any  $f,g\in K[x,y]$ we have $\Jac(\alpha(f),\alpha(g)) = \alpha(\Jac(f,g)) \Jac(\alpha)$, because
\begin{multline*}
\frac{\partial}{\partial x}\alpha(f) = \frac{\partial f}{\partial x}(\alpha (x),\alpha(y)) \frac{\partial\alpha(x)}{\partial x}
+\frac{\partial f}{\partial y}(\alpha (x),\alpha(y))\frac{\partial\alpha(y)}{\partial x} \\  
= \alpha(\frac{\partial f}{\partial x})\frac{\partial\alpha(x)}{\partial x} + \alpha(\frac{\partial f}{\partial y})\frac{\partial\alpha(y)}{\partial x}.
\end{multline*}
It follows that $\{\alpha(f),\alpha(g)\} = \alpha(\{f,g\}) j(\alpha)$. This shows that $\SAutK \subset \SAutLA(P)$.

(b) Now let $\phi\in\SAutLA(P)$. Then $j(\phi(x),\phi(y)) = \{\phi(x),\phi(y)\} = \phi(1)=1$ and, by Lemma~\ref{lem2}, $K[\phi(x),\phi(y)] = K[x,y]$. Hence, we can define an automorphism $\alpha\in\SAutK$ by $\alpha(x):=\phi(x)$ and $\alpha(y) := \phi(y)$. From (a) we see that $\alpha\in\SAutLA(P)$, and from Lemma~\ref{lem2} we get $\phi=\alpha$, hence $\phi \in\SAutK$.
\end{proof}

\begin{rem}\lab{homP.rem}
The first part of the proof above shows the following. If $f,g\in P$ are such that $\{f,g\} = 1$, then the $K$-algebra homomorphism defined by $x\mapsto f$ and $y \mapsto g$ is an injective homomorphism of $P$ as a Lie algebra. (Injectivity follows, because $f,g$ are algebraically independent.)
\end{rem}

\subsection*{Lie subalgebras of $P$}
The subspace 
$$
P_{\leq 2} := K\oplus P_{1}\oplus P_{2} = K \oplus Kx \oplus Ky \oplus Kx^{2} \oplus Kxy \oplus Ky^{2} \subset P
$$ 
is a Lie subalgebra. 
This can be deduced from the following Lie brackets which we note here for later use.
\begin{gather}
\{x^{2},xy\} = 2x^{2}, \ \{x^{2},y^{2}\}=4xy, \ \{y^{2},xy\}= -2y^{2};\\
\{x^{2},x\}=0, \ \{xy,x\}=-x, \ \{y^{2},x\}=-2y, \\ \{x^{2},y\}=2x, \ \{xy,y\}=y, \ \{y^{2},y\}=0;\\
\{x,y\}=1.
\end{gather}
For example, $P_{2}=Kx^{2}\oplus Kxy \oplus Ky^{2}$ is a Lie subalgebra of $\Ptwo$ isomorphic to $\sltwo$, and $P_{1}=Kx \oplus Ky$ is the two-dimensional simple $P_{2}$-module.

From Remark~\ref{homP.rem} we get the following lemma.
\begin{lem} 
Let $f,g\in K[x,y]$ such that $\{f,g\}=1$. Then $\langle 1,f,g,f^{2},fg,g^{2}\rangle \subset P$ is a Lie subalgebra isomorphic to $P_{\leq 2}$. An isomorphism is induced from the $K$-algebra homomorphism $P \to P$ defined by $x\mapsto f, y\mapsto g$.
\end{lem}
\begin{defn} For $f,g \in K[x,y]$ such that $\{f,g\} \in K^{*}$ we put 
$$
P_{f,g}:= \langle 1,f,g,f^{2},fg,g^{2}\rangle \subset P.
$$
We have just seen that this is a Lie algebra isomorphic to $\Ptwo$. Clearly, $P_{f,g}=P_{f_{1},g_{1}}$ if $\langle 1,f,g\rangle = \langle 1,f_{1},g_{1}\rangle$. Denoting by $\rad L$ the solvable radical of the Lie algebra $L$ we get
$$
\rad P_{f,g} = \langle 1, f, g\rangle \text{ \ and \ } P_{f,g}/\rad P_{f,g} \simeq \sltwo.
$$
\end{defn}

\begin{prop}\lab{subP.prop}
Let $Q \subset P$ be a Lie subalgebra isomorphic to $\Ptwo$. Then $K \subset Q$, and  
$Q = P_{f,g}$ for every pair $f,g\in L$ such that $\langle 1, f, g\rangle=\rad Q$. In particular, $\{f,g\}\in K^{*}$. 
\end{prop}
\begin{proof}
We first show that $\cent(Q) =K$. In fact, $Q$ contains elements $f,g$ such that $\{f,g\}\neq 0$. If $h\in\cent(Q)$, then $h \in \cent_{P}(f)\cap \cent_{P}(g) = K$, by Lemma~\ref{lem1}(c).

Now choose an isomorphism $\phi\colon \Ptwo \simto Q$. Then $\phi(K)=K$, and replacing $\phi$ by $\phi\circ s(t)$ with a suitable $t\in K^{*}$ we can assume that $\phi(1)=1$. Setting $f:=\phi(x), g:=\phi(y)$ we get $\{f,g\} = 1$, and putting $f_{0}:=\phi(x^{2}), f_{1}:=\phi(xy), f_{2}:=\phi(y^{2})$ we find
$$
\{f_{1},f\} = \phi \{xy,x\}=\phi(-x)=-f = \{fg,f\}.
$$
Similarly, $\{f_{1},g\}=\{fg,g\}$, hence $fg = f_{1}+ c \in Q$, by Lemma~\ref{lem1}(c). Next we have
$$
\{f_{0},f\} =0 \text{ \ and \ } \{f_{0},g\}=\phi(\{x^{2},y\})=\phi(2x)=2f = \{f^{2},g\}.
$$
Hence $f^{2}=f_{0}+d$, and thus $f^{2}\in Q$.
A similar calculation shows that $g^{2}\in Q$, so that we finally get  $Q = P_{f,g}$. 
\end{proof}

\subsection*{Characterization of $\Ptwo$}
The following lemma gives a characterization of the Lie algebras isomorphic to  $\Ptwo$.

\begin{lem}\lab{charP.lem}
Let $Q$ be a Lie algebra containing a subalgebra $Q_{0}$ isomorphic to $\sltwo$. Assume that
\be
\item $Q = Q_{0}\oplus V_{2}\oplus V_{1}$ as an $Q_{0}$-module where $V_{i}$ is simple of dimension $i$,
\item $V_{1}$ is the center of $Q$, and 
\item $[V_{2},V_{2}] = V_{1}$.
\ee
Then $Q$ is isomorphic to $\Ptwo$.
\end{lem}

\begin{proof}
Choosing an isomorphism of $P_{2}=\langle x^{2},xy,y^{2}\rangle$ with $Q_{0}$ we find a basis $(a_{0},a_{1},a_{2})$ of $Q_{0}$ with relations
\[\tag{$1'$}
[a_{0},a_{1}]= 2a_{0}, \ [a_{0},a_{2}]= 4a_{1}, \ [a_{2},a_{1}]= -2a_{2}
\]
(see (1) above). Since $V_{2}$ is a simple two-dimensional $Q_{0}$-module and $Kx\oplus Ky$ a simple two-dimensional $P_{2}$-module we can find a basis $(b,c)$ of $V_{2}$ such that
\begin{gather*}\tag{$2'$}
[a_{0},b]=0, \ [a_{1},b]=-b, \ [a_{2},b] = -2c, \\ \tag{$3'$}
[a_{0},c]=2b, \ [a_{1},c]=c, \ [a_{2},c] = 0
\end{gather*}
(see (2) and (3) above). Finally, the last assumption (c) implies that 
\[\tag{$4'$}
d:=[b,c]\neq 0, \text{ hence } V_{1}=Kd.
\]
Comparing the relations (1)--(4) with ($1'$)--($4'$) we see that the linear map $\Ptwo \to Q$ given by $x^{2}\mapsto a_{0}$, $xy\mapsto a_{1}$, $y^{2}\mapsto a_{2}$, $x\mapsto b$, $y\mapsto c$, $1 \mapsto d$ is a Lie algebra isomorphism.
\end{proof}

\medskip
\section{Vector fields on affine 2-space}
\subsection*{The action of $\AutK$ on vector fields}
The group $\AutK$ acts on the vector fields $\VF(\Atwo)$ by conjugation: If $D$ is a derivation and $\alpha\in\AutK$, then $\alpha(D) := \alpha\circ D \circ \alpha^{-1}$. Writing $D = p\dx+q\dy$ and $\alpha=(f,g)$, then
\[\tag{$*$}
\alpha(D) = \frac{1}{j(\alpha)}\left(\left(g_{y}\alpha(p)-f_{y}\alpha(q)\right)\dx + \left(-g_{x}\alpha(p)+f_{x}\alpha(q)\right)\dy\right)
\]
In fact, $\alpha(\Jac(\alpha^{-1})) \cdot \Jac(\alpha) = E$, hence $\alpha(\Jac(\alpha^{-1})) = 
\Jac(\alpha)^{-1}= \frac{1}{j(\alpha)} \begin{bmatrix} g_{y} & -f_{y}\\ -g_{x} & f_{x} \end{bmatrix}$. 
Thus we get  for $h\in K[x,y]$ 
\begin{multline*}
\alpha(D)(h) = \alpha(D(\alpha^{-1}(h))) = (h_{x},h_{y})\cdot \alpha(\Jac(\alpha^{-1})) \cdot\begin{bmatrix} \alpha(p) \\ \alpha(q) \end{bmatrix} =\\= (h_{x},h_{y})\cdot \frac{1}{j(\alpha)} \begin{bmatrix} g_{y} & -f_{y}\\ -g_{x} & f_{x} \end{bmatrix}
\cdot \begin{bmatrix} \alpha(p) \\ \alpha(q) \end{bmatrix}
\end{multline*}
In particular, for $h=x$ or $h=y$, we find
$$
\alpha(D)(x) =  \frac{1}{j(\alpha)}(g_{y}\alpha(p) - f_{y}\alpha(q))\text{ \ and \ }
\alpha(D)(y) =  \frac{1}{j(\alpha)}(-g_{x}\alpha(p) + f_{x}\alpha(q)),
$$
and the claim follows. 
\begin{rem}\lab{etale.rem}
If $\alpha\colon K[x,y] \to K[x,y]$ is \'etale, i.e. $j(\alpha)\in K^{*}$, then formula $(*)$ still makes sense and defines a map 
$$
\alpha \colon \VF(\Atwo) \to \VF(\Atwo), \ D \mapsto \alpha(D),
$$ 
which is an injective homomorphism of Lie algebras. In fact, we have by definition $\alpha(D)\circ\alpha = \alpha\circ D$, and so
\begin{multline*}
\alpha([D_{1},D_{2}])\circ \alpha = \alpha\circ D_{1}\circ D_{2}-\alpha\circ D_{2}\circ D_{1}= \alpha(D_{1})\circ \alpha\circ D_{2} - 
\alpha(D_{2})\circ \alpha\circ D_{1}=\\
=\alpha(D_{1})\circ\alpha(D_{2})\circ \alpha - \alpha(D_{2})\circ\alpha(D_{1})\circ \alpha = [\alpha(D_{1}),\alpha(D_{2})]\circ \alpha.
\end{multline*}
\end{rem}
Recall that  $\VFd(\Atwo) \subset \VF(\Atwo)$ are the vector fields $D$ with $\Div D \in K$. Clearly,  the divergence $\Div\colon \VFd(\Atwo) \to K$ is a character with kernel $\VFzero(\Atwo)$, and we have the decomposition
$$
\VFd(\Atwo) = \VFzero(\Atwo) \oplus KE \text{ \ where \ } E:=x\dx + y\dy \text{ \ is the {\it Euler field}}.
$$

\begin{lem}\lab{equiv.lem}
If $\alpha\colon K[x,y] \to K[x,y]$ is \'etale, then $\alpha(D_{h}) = j(\alpha)^{-1}D_{\alpha(h)}$, and $\Div(\alpha(E)) = 2$. Hence $\alpha(\VFzero(\Atwo)) \subset \VFzero(\Atwo)$ and $\alpha(\VFd(\Atwo)) \subset \VFd(\Atwo)$ . In particular, the homomorphism
$\mu\colon P \to \VF(\Atwo)$ is equivariant with respect to the group $\SAutK=\SAutLA(P)$.
\end{lem}
\begin{proof} We have $\alpha(D_{h})\circ\alpha = \alpha\circ D_{h}$, hence 
\begin{multline*}
\alpha(D_{h})(\alpha(f)) = \alpha(D_{h}(f)) = \alpha(j(h,f)) = j(\alpha)^{-1}j(\alpha(h),\alpha(f)) =\\= j(\alpha)^{-1}D_{\alpha(h)}(\alpha(f)).
\end{multline*}
From formula $(*)$ we get $\alpha(E) = \frac{1}{j(\alpha)}\left((g_{y}f-f_{y}g)\dx + (-g_{x}f+f_{x}g)\dy\right)$ which implies that $\Div\alpha(E) = 2$.
\end{proof}

\subsection*{Lie subalgebras of $\VF(\Atwo)$}
Let $\Aff(\Atwo)$ denote the group of {\it affine transformations\/} of $\Atwo$, $x \mapsto Ax + b$, where $A\in\GL_{2}(K)$ and $b\in K^{2}$. The determinant defines a character $\det\colon \Aff(\Atwo) \to K^{*}$ whose kernel will be denoted by $\SAff(\Atwo)$. For the corresponding Lie algebras we write $\saff_{2}:=\Lie \SAff(\Atwo) \subset \aff_{2}:=\Lie \Aff(\Atwo)$. There is a canonical embedding $\aff_{2}\subset \VF(\Atwo)$ which identifies $\aff_{2}$ with the Lie subalgebra
$$
\langle \dx,\dy, x\dx + y\dy, x\dx - y\dy, x\dy, y\dx \rangle \subset \VFd(\Atwo),
$$
and $\saff_{2}$ with
$$
\mu(P_{x,y}) = \langle \dx,\dy, x\dx - y\dy, x\dy, y\dx \rangle \subset \VFzero(\Atwo).
$$
Note that the {\it Euler field\/} $E=x\dx + y\dy \in \aff_{2}$ is determined by the condition that $E$ acts trivially on $\sltwo$ and that $[E,D]=-D$ for $D\in \rad(\saff_{2})=K\dx\oplus K\dy$. We also remark that the centralizer of $\saff_{2}$ in $\VF(\Atwo)$ is trivial: 
$$
\cent_{\VF(\Atwo)}(\saff_{2})=(0).
$$
In fact, $\cent_{\VF(\Atwo)}(\{\dx,\dy)\}= K\dx\oplus K\dy$, and $(K\dx\oplus K\dy)^{\sltwo} = (0)$.

Let $\alpha=(f,g)\in\EndK$ be \'etale, and assume, for simplicity, that $j(f,g)=1$. Then we get from formula $(*)$ 
\begin{gather*}
\alpha(\dx)=g_{y}\dx -g_{x}\dy = - D_{g}, \quad
\alpha(\dy)= -f_{y}\dx +f_{x}\dy = D_{f},\\
\alpha(x\dy) = f D_{f}=\textstyle{\frac{1}{2}}D_{f^{2}}, \quad \alpha(y \dx) = -g D_{g}= -\textstyle{\frac{1}{2}}D_{g^{2}}, \\ 
\alpha(x\dx) = -fD_{g}, \quad \alpha(y\dy)=g D_{f}, \quad \alpha(x\dx-y\dy) = -D_{fg}. 
\end{gather*}
This shows that for an \'etale map $\alpha=(f,g)$ we obtain
\begin{gather*}
\alpha(\aff_{2}) = \langle D_{f},D_{g},D_{f^{2}},D_{g^{2}},fD_{g},gD_{f}\rangle,  \\
\alpha(\saff_{2}) = \langle D_{f},D_{g},D_{f^{2}},D_{g^{2}},D_{fg}\rangle = \mu(P_{f,g})
\end{gather*}

\begin{prop}\lab{subVec1.prop}
Let $L\subset \VFd(\Atwo)$ be a Lie subalgebra isomorphic to $\saff_{2}$. Then there is an \'etale map $\alpha=(f,g)$ such that $L = \alpha(\saff_{2})$. More precisely, if $(D_{f},D_{g})$ is a basis of $\rad(L)$, then  $L = \langle D_{f},D_{g},D_{f^{2}},D_{g^{2}},D_{fg}\rangle$, and one can take $\alpha=(f,g)$. 
\end{prop}

\begin{proof}
We first remark that $L \subset \VFzero(\Atwo)$, because $\saff_{2}$ has no non-trivial character.
By Proposition~\ref{subP.prop} it suffices to show that $Q:=\mu^{-1}(L)\subset P$ is isomorphic to $\Ptwo$. We fix a decomposition $L =L_{0}\oplus \rad(L)$ where $L_{0}\simeq \sltwo$.
It is clear that the Lie subalgebra $\tilde Q:=\mu^{-1}(L_{0})\subset P$ contains a copy of $\sltwo$, i.e. $\tilde Q= Q_{0}\oplus K$ where $Q_{0}\simeq \sltwo$. Hence, as an $Q_{0}$-module, we get $Q = Q_{0}\oplus V_{2}\oplus K$ where $V_{2}$ is a two-dimensional irreducible $Q_{0}$-module which is isomorphically mapped onto $\rad(L)$ under $\mu$.  Since $\{\rad(L),\rad(L)\} = (0)$ we have $\{V_{2},V_{2}\}\subset K$. Now the claim follows from Lemma~\ref{charP.lem} if we show that $\{V_{2},V_{2}\}\neq (0)$. 

Assume that $\{V_{2},V_{2}\}= (0)$. Choose an $\sltwo$-triple $(e_0,h_0,f_0)$ in $Q_{0}$ and a basis $(f,g)$ of $V_{2}$ such that $\{e_0,f\}=g$ and $\{e_0,g\}=0$. Since $\{f,g\}=0$ we get from Lemma~\ref{lem1}(b) that $f,g\in K[h]$ for some $h\in K[x,y]$, i.e. $f=p(h)$ and $g=q(h)$ for some polynomials $p,q\in K[t]$. But then $0=\{e_0,g\}=\{e_0,q(h)\} = q'(h)\{e_0,h\}$ and so $\{e_0,h\}=0$. This implies that $g=\{e_0,f\}=\{e_0,p(h)\}=p'(h)\{e_0,h\}=0$, a contradiction.
\end{proof}

\begin{rem}\lab{splitsaff.rem}
The above description of the Lie subalgebras $L$ isomorphic to $\saff_{2}$ also gives a Levi decomposition of $L$. In fact,
$(D_{f},D_{g})$ is a basis of $\rad(L)$ and $L_{0} := \langle  D_{f^{2}},D_{g^{2}},D_{fg}\rangle$ is a subalgebra isomorphic to $\sltwo$. The following corollary shows that every Levi decomposition is obtained in this way.
\end{rem}

\begin{cor}\lab{subVec.cor}
Let $L \subset \VFd(\Atwo)$ be a Lie subalgebra isomorphic to $\saff_{2}$, and let $L = \rad(L)\oplus L_{0}$ be a Levi decomposition. Then there exist $f,g\in K[x,y]$ such that $\rad(L) = \langle D_{f},D_{g}\rangle$ and $L_{0}= \langle D_{f^{2}},D_{fg},D_{g^{2}}\rangle$.
Moreover, if $L'\subset \VFd(\Atwo)$ is another Lie subalgebra isomorphic to $\saff_{2}$ and if $L'\supset L_{0}$, then $L'= L$.
\end{cor}
\begin{proof}
We can assume that $L = \saff_{2} =\langle D_{x},D_{y},D_{x^{2}},D_{y^{2}},D_{xy}\rangle$. Then every Lie subalgebra $L_{0}\subset L$ isomorphic to $\sltwo$ is the image of $\sltwo=  \langle D_{x^{2}},D_{y^{2}},D_{xy}\rangle$ under conjugation with an element $\alpha$ of the solvable radical $R$ of $\SAff_{2}$. As a subgroup of $\AutK$ the elements of $R$ are the translations $\alpha=(x+a,y+b)$, and we get $\rad(L)=\langle D_{x+a},D_{y+b}\rangle$ and $\alpha(\sltwo)=\langle D_{(x+a)^{2}},D_{(y+b)^{2}},D_{(x+a)(y+b}\rangle$ as claimed. 

For the last statement, we can assume that $L' = \langle D_{f},D_{g},D_{f^{2}},D_{g^{2}},D_{fg}\rangle$ such that $\langle D_{f^{2}},D_{g^{2}},D_{fg}\rangle = \sltwo$. This implies that $\langle f^{2},g^{2},fg,1\rangle=\langle x^{2},y^{2},xy,1\rangle$, and the claim follows.
\end{proof}

\begin{prop}\lab{subVec2.prop}
Let $M\subset \VFd(\Atwo)$ be a Lie subalgebra isomorphic to $\aff_{2}$. Then there is an \'etale map $\alpha$ such that $M = \alpha(\aff_{2})$. More precisely, if $(D_{f},D_{g})$ is a basis of $\rad([M,M])$, then  $M = \langle D_{f},D_{g},fD_{f},gD_{g},gD_{f},fD_{g}\rangle$, and one can take $\alpha=(f,g)$. 
\end{prop}
\begin{proof}
The subalgebra $M':=[M,M]$ is isomorphic to $\saff$, hence, by Proposition~\ref{subVec1.prop}, $M'=\alpha(\saff_{2})$ for an \'etale map $\alpha=(f,g)$ where we can assume that $j(\alpha)=1$. We want to show that $\alpha(\aff_{2})=M$.
Consider the decomposition 
$M = J \oplus M_0 \oplus KD$
where $J = \rad(M')$, $M_0$ is isomorphic to $\sltwo$, and $D$ is the Euler-element acting
trivially on $M_0$. 
We have $\alpha(\aff_2)=M' \oplus KE$ where $E$ is the image of the Euler element of $\aff_2$.
Since $\VFd(\Atwo) = \VFzero(\Atwo) \oplus KD'$ for any $D' \in \VFd(\Atwo)$ with $\Div D'\neq 0$ we can write $D = aE + F$ with some $a\in K$ and $F \in \VFzero(\Atwo)$, i.e. $F = D_{h}$ for some $h\in K[x,y]$.

By construction, $F=D-aE$ commutes with $M_0$.
Since  $M_{0}= \langle D_{f^{2}},D_{g^{2}},D_{fg} \rangle$ we get
   $\{h,f^2\}  =  c$ where  $c \in K$.
Thus
$c = \{h,f^2\} = 2f\{h,f\}$
which implies that $\{h,f\}=0$. Similarly, we find $\{h,g\}=0$, hence $h$ is in the center of $\mu^{-1}(M')=P_{f,g} \subset P$. 
Thus, by Lemma~\ref{lem1}(c), $h\in K$ and so $D_h=0$ which implies $D = aE$.
\end{proof}

\medskip
\section{Vector fields and the Jacobian Conjecture}
\subsection*{The Jacobian Conjecture} 
Recall that the {\it Jacobian Conjecture} in dimension $n$ says that an \'etale homomorphism $\alpha\colon K[x_{1},\ldots,x_{n}] \to K[x_{1},\ldots,x_{n}]$ is an isomorphism.

\begin{thm}\lab{JC.thm}
The following statements are equivalent. 
\be
\item[(i)] The Jacobian Conjecture holds in dimension 2.
\item[(ii)] All Lie subalgebras of $P$ isomorphic to $\Ptwo$ are equivalent under $\AutLA(P)$.
\item[(iii)] All Lie subalgebras of $\VFd(\Atwo)$ isomorphic to $\saff_{2}$ are conjugate under $\AutK$.
\item[(iv)] All Lie subalgebras of $\VFd(\Atwo)$ isomorphic to $\aff_{2}$ are conjugate under $\AutK$.
\ee
\end{thm}
For the proof we need to compare the automorphisms of $P$ with those of the image $\mu(P)=\VFzero(\Atwo) \simeq P/K$. Since $K$ is the center $P$, we have a canonical homomorphism $F\colon\AutLA(P) \to \AutLA(P/K)$, $\phi\mapsto \bar\phi$.
\begin{lem}
The map $F\colon \AutLA(P) \to \AutLA(P/K)$ is an isomorphism.
\end{lem}
\begin{proof}  If $\phi\in\ker F$, then $\phi(x)=x+a, \phi(y)=y+b$ where $a,b\in K$. By Lemma~\ref{aut.lem}, the $K$-algebra automorphisms $\alpha$ of $K[x,y]$ defined by $x\mapsto x+a$, $y \mapsto y+b$ is a Lie algebra automorphisms of $P$, and $\phi=\alpha$ by Lemma~\ref{lem2}. But then $\phi(x^{2}) = (x+a)^{2}= x^{2}+2ax + a^{2}$, and so $\bar\phi(\overline{x^{2}}) = \overline{x^{2}} + 2a\overline{x}$. Therefore, $a=0$, and similarly we get $b=0$, hence $\phi=\id_{P}$.

Put $\Pb := P/K$ and  let $\rho\colon \Pb \simto \Pb$ be a Lie algebra automorphism. Then $\bar L := \rho(\Pbtwo)\subset \Pb$ is a Lie subalgebra isomorphic to $\saff_{2}$ and thus $L:=p^{-1}(\bar L)$ is a Lie subalgebra of $P$ isomorphic to $\Ptwo$, by Proposition~\ref{subP.prop}. Choose  $f,g\in L$ such that $\bar f=\phi(\bar x)$ and $\bar g = \phi(\bar y)$. Then $\langle 1,f,g \rangle = \rad(L)$, and so $L=P_{f,g}$, by Proposition~\ref{subP.prop}. It follows that the map $\phi\colon x \mapsto f, y \mapsto g$ is an injective endomorphism of $P$ (see Remark~\ref{homP.rem}), and $\bar\phi = \rho$. Since $\rho$ is an isomorphism the same holds for $\phi$.
\end{proof}
\subsection*{Proof of Theorem~\ref{JC.thm}}
(i)$\Rightarrow$(ii): If $L \subset P$ is isomorphic to $\Ptwo$, then $L = P_{f,g}$ for some $f,g\in K[x,y]$ such that $\{f,g\}=1$ (Proposition~\ref{subP.prop}). By (i) we get $K[x,y] = K[f,g]$, and so the endomorphism $x \mapsto f, y\mapsto g$ of $K[x,y]$ is an isomorphism of $P$, mapping $\Ptwo$ to $L$.

(ii)$\Rightarrow$(iii): 
If $\bar L \subset \VF^{c}(\Atwo)$ is a Lie subalgebra isomorphic to $\saff_{2}$, then $\bar L = \mu(P_{f,g})$ for some $f,g \in K[x,y]$, by Proposition~\ref{subVec1.prop}. By (ii), $P_{f,g} = \alpha(\Ptwo)$ for some $\alpha\in\SAutLA(P)=\SAutK$. Hence $\bar L =   \mu( \alpha(\Pbtwo))=\bar\alpha(\saff_{2})$, by Lemma~\ref{equiv.lem}.

(iii)$\Rightarrow$(iv):
Let $M \subset \VFd(\Atwo)$ be a Lie subalgebra isomorphic to $\aff_{2}$, and set $M':=[M,M]\simeq \saff_{2}$. By (iii) there is an automorphism $\alpha\in\AutK$ such that $M'=\alpha(\saff_{2})$. It follows that $\alpha(\aff_{2}) = M$ since $M$ is determined by $\rad(M')$ as a Lie subalgebra, by Proposition~\ref{subVec2.prop}.

(iv)$\Rightarrow$(i): Let $f,g\in K[x,y]$ such that $\{f,g\}=1$, and let $\alpha\colon K[x,y] \to K[x,y]$ be the \'etale homomorphism defined by $\alpha(x)=f$ and $\alpha(y)=g$. Then $M:=\alpha(\aff_{2})\subset \VFd(\Atwo)$ is a Lie subalgebra isomorphic to $\aff_{2}$, by Lemma~\ref{equiv.lem}. By assumption (iv), there is an automorphism $\beta \in\AutK$ such that $\beta(\aff_{2}) = M$. It follows that $\beta^{-1}\circ \alpha$ is an \'etale homomorphism which induces an isomorphism of $\aff_{2}$, hence of $\saff_{2}$ and thus of $\rad(\saff_{2}) = K\dx \oplus K\dy$. This implies that $\beta^{-1}\circ\alpha$ is an automorphism, and so $K[f,g]=K[x,y]$.
\qed
\begin{rem}\lab{example.rem}
It is not true that the Lie subalgebras of $P$ or of $\VFd(\Atwo)$ isomorphic to $\sltwo$ are equivalent respectively conjugate. This can be seen from the example $S = Kx^{2}y\oplus Kxy \oplus Ky \subset P$ which is isomorphic to $\sltwo$, but not equivalent to $Kx^{2}\oplus Kxy \oplus Ky^{2}$ under $\AutLA(P)$. In fact, the element $x^{2}y$ does not act locally finitely on $P$.
\end{rem}

\subsection*{Algebraic Lie algebras}
If an algebraic group $G$ acts on an affine variety $X$ we get a canonical anti-homomorphism of Lie algebras $\Phi\colon\Lie G \to \VF(X)$  defined in the usual way:
$$
\Lie G \ni A \mapsto \xi_{A} \text{ with } (\xi_{A})_{x} := d\phi_{x}(A) \text{ for } x\in X,
$$
where $\phi_{x}\colon G \to X$ is the orbit map $g\mapsto gx$. A Lie algebra $L \subset \VF(X)$ is called {\it algebraic\/} if $L$ is contained in  $\Phi(\Lie G)$ for some action of an algebraic group $G$ on $X$. If is shown in \cite{CoDr2003From-Lie-algebras-} that $L$ is algebraic if and only if $L$ acts locally finitely on $\VF(X)$. With this result we get the following consequence of our Theorem 1.
\begin{cor}\lab{corollary} The following statements are equivalent.
\be
\item[(i)] The Jacobian Conjecture holds in dimension 2.
\item[(ii)] All Lie subalgebras of $\VFd(\Atwo)$ isomorphic to $\saff_{2}$ are algebraic.
\item[(iii)] All Lie subalgebras of $\VFd(\Atwo)$ isomorphic to $\aff_{2}$ are algebraic.
\ee
\end{cor}
\begin{proof}
It is clear that the equivalent statements (i), (ii) or (iii) of Theorem 1 imply (ii) and (iii) from the corollary. It follows from the Propositions~\ref{subVec1.prop} and \ref{subVec2.prop} that every Lie subalgebra $L$ isomorphic to $\saff_{2}$ is contained in a Lie subalgebra $Q$ isomorphic to $\aff_{2}$, hence (iii) implies (ii). It remains to prove that (ii) implies (i).

We will show that (ii) implies that $L$ is equivalent to $\saff_{2}$. Then the claim follows from Theorem 1.  By (ii), there is a connected algebraic group $G$ acting faithfully on $\Atwo$ such that $\Phi(\Lie G)$ contains $L$. Therefore, $\Lie G$ contains a subalgebra $\s$ isomorphic to $\sltwo$, and so  $G$ contains a closed subgroup $S$ such that $\Lie S = \s$. Since every action of $\SLtwo$ on $\Atwo$ is linearizable (see \cite{KrPo1985Semisimple-group-a}), there is an automorphism $\alpha$ such that $\alpha(\s) = \sltwo = \langle x\dy,y\dx,x\dx - y\dy \rangle$. But this implies, by Corollary~\ref{subVec.cor}, that $\alpha(L) =\saff_{2}$. 
\end{proof}



\begin{thebibliography}{BCW82}

\bibitem[CD03]{CoDr2003From-Lie-algebras-}
Arjeh~M. Cohen and Jan Draisma.
\newblock From {L}ie algebras of vector fields to algebraic group actions.
\newblock {\em Transform. Groups}, 8(1):51--68, 2003.

\bibitem[Kam79]{Ka1979Automorphism-group}
T.~Kambayashi, \emph{Automorphism group of a polynomial ring and algebraic
  group action on an affine space}, J. Algebra \textbf{60} (1979), no.~2,
  439--451. 
  
\bibitem[Kra96]{Kr1996Challenging-proble}
Hanspeter Kraft, \emph{Challenging problems on affine {$n$}-space},
  Ast{\'e}risque (1996), no.~237, Exp.\ No.\ 802, 5, 295--317, S{{\'e}}minaire
  Bourbaki, Vol. 1994/95. 


\bibitem[KP85]{KrPo1985Semisimple-group-a}
Hanspeter Kraft and Vladimir~L. Popov.
\newblock Semisimple group actions on the three-dimensional affine space are
  linear.
\newblock {\em Comment. Math. Helv.}, 60(3):466--479, 1985.


\bibitem[NN88]{NoNa1988Rings-of-constants}
Andrzej Nowicki and Masayoshi Nagata.
\newblock Rings of constants for {$k$}-derivations in {$k[x_1,\cdots,x_n]$}.
\newblock {\em J. Math. Kyoto Univ.}, 28(1):111--118, 1988.


\bibitem[Sch89]{Sch89}  Schwarz, G. W., 
\newblock Exotic algebraic group actions., 
\newblock {\em J. C. R. Acad. Sci. Paris} 309 (1989), 89-94.

\bibitem[Kum02]{Kum02} Shrawan Kumar, 
\newblock Kac-Moody groups, their flag varieties and representation theory.  
\newblock {\em J. Progress in Mathematics, vol. 204, Birkh ̈auser Boston Inc., Boston, MA,}, 2002.

\bibitem [Sh81]{Sh81} I.R. Shafarevich,
\newblock On some infinite-dimensional groups. II,
\newblock {\em  Izv. Akad. Nauk SSSR Ser. Mat.} 45 (1981), no. 1, 214-226, 240.   
  
\end{thebibliography}
\end{document}